\documentclass[11pt]{article}
\usepackage[utf8]{inputenc}
\pdfoutput=1
\usepackage{amsfonts,amsmath}
\usepackage{amssymb}
\usepackage{url}
\usepackage{amsthm,amsmath,amscd}
\usepackage{enumerate}
\usepackage{lipsum}
\usepackage[margin=25pt,font=small,labelfont=bf]{caption}
\usepackage{hyperref}
\usepackage[pdftex]{graphicx}
\usepackage{graphicx}
\usepackage[square,numbers,sort&compress]{natbib}

\usepackage[margin=1in]{geometry}
\newcommand{\RR}{\mathbb R}
\newcommand{\PP}{\mathbb P}
\newcommand{\EE}{\mathbb E}

\newcommand{\bna}{\begin{eqnarray}}
\newcommand{\ena}{\end{eqnarray}}
\newcommand{\ba}{\begin{eqnarray*}}
\newcommand{\ea}{\end{eqnarray*}}
\newcommand{\bs}[1]{}
\newtheorem{theorem}{Theorem}[section]
\newtheorem{lemma}[theorem]{Lemma}
\newtheorem{proposition}[theorem]{Proposition}
\newtheorem{corollary}[theorem]{Corollary}

\newtheorem{remark}[theorem]{Remark}
\newtheorem{definition}[theorem]{Definition}

\newcommand{\QQ}{\mathbb Q}
\def\p{{\bf p}}

\def\x{{\bf x}}
\def\t{{\bf t}}
\def\l{{\bf l}}
\def\I{{\bf I}}

\def\y{{\bf y}}
\def\e{{\bf e}}
\def\0{{\bf 0}}

\def\v{{\bf v}}
\def\Q{{\bf Q}}
\newcommand{\A}{{\bf A}}
\newcommand{\QQQ}{{\cal Q}}
\newcommand{\SSS}{{\cal S}}

\def\q{{\bf q}}

\begin{document}
\title{Affine Rigidity and Conics at Infinity}

\author{Robert Connelly
\and
Steven J. Gortler
\and
Louis Theran}
\date{}

\maketitle

\begin{abstract}
We prove that if a framework of a graph is neighborhood affine rigid
in $d$-dimensions (or has the stronger property of having
an equilibrium stress matrix of rank
$n-d-1$)
then it has an affine flex
(an affine, but non Euclidean, transform of space that preserves
all of the edge lengths)
if and only if the framework is ruled on a single quadric.
This strengthens
and also simplifies a related result by Alfakih.
It also allows us to prove that the property of super stability is
invariant with respect to projective transforms and also to the coning
and slicing operations.
Finally this allows us to unify some previous results on the
Strong Arnold Property of matrices.
\end{abstract}

\section{Introduction}

Let $G$ be a connected graph with $n$ labeled vertices and $m$ edges
and $\p$ be a configuration of the $n$ vertices in $\EE^d$.
Throughout, we will assume that $\p$ has a full $d$ dimensional affine span.
The pair
$(G,\p)$ is called a framework in $\EE^d$.
Two frameworks $(G,\p)$ and $(G,\q)$ are called (Euclidean) equivalent if they
share
the same $m$ lengths measured along the
edges in $G$.
Two frameworks $(G,\p)$ and $(G,\q)$ are called (Euclidean) congruent if they
are
related through a $d$-dimensional
Euclidean transform.

In rigidity theory, one
may be interested in knowing if there is a second framework $(G,\q)$ in
$\EE^d$ that is equivalent to, but not congruent to $(G,\p)$.
If there is no such $(G,\q)$ in $\EE^d$,
we say that $(G,\p)$ is \emph{globally rigid} in $\EE^d$.
If there is no such $(G,\q)$ in
\emph{any} dimension, we say that $(G,\p)$ is \emph{universally rigid}.

\begin{figure}[h]
\begin{center}
\includegraphics[width=.3\textwidth]{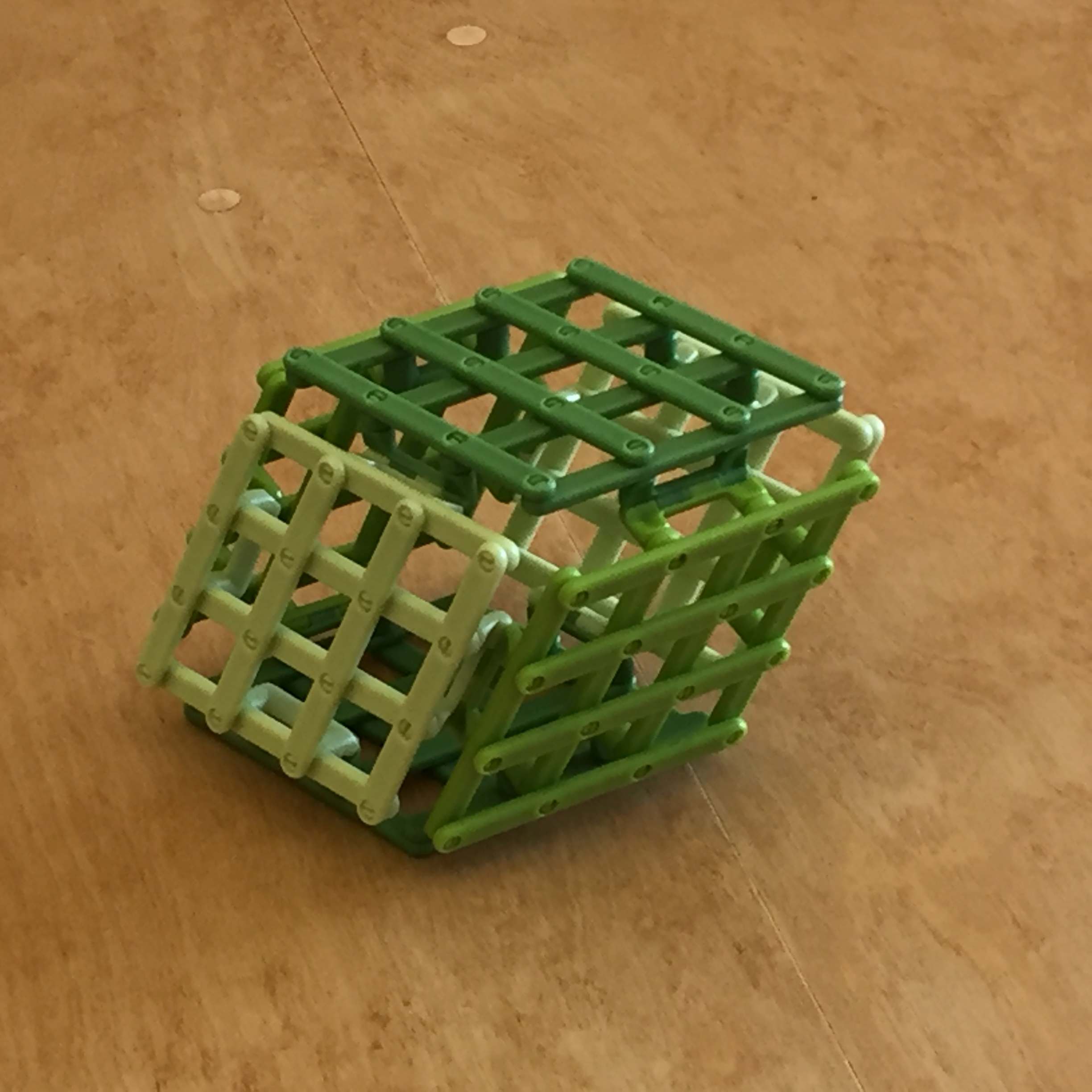}%
\hspace{1 in}
\includegraphics[width=.3\textwidth]{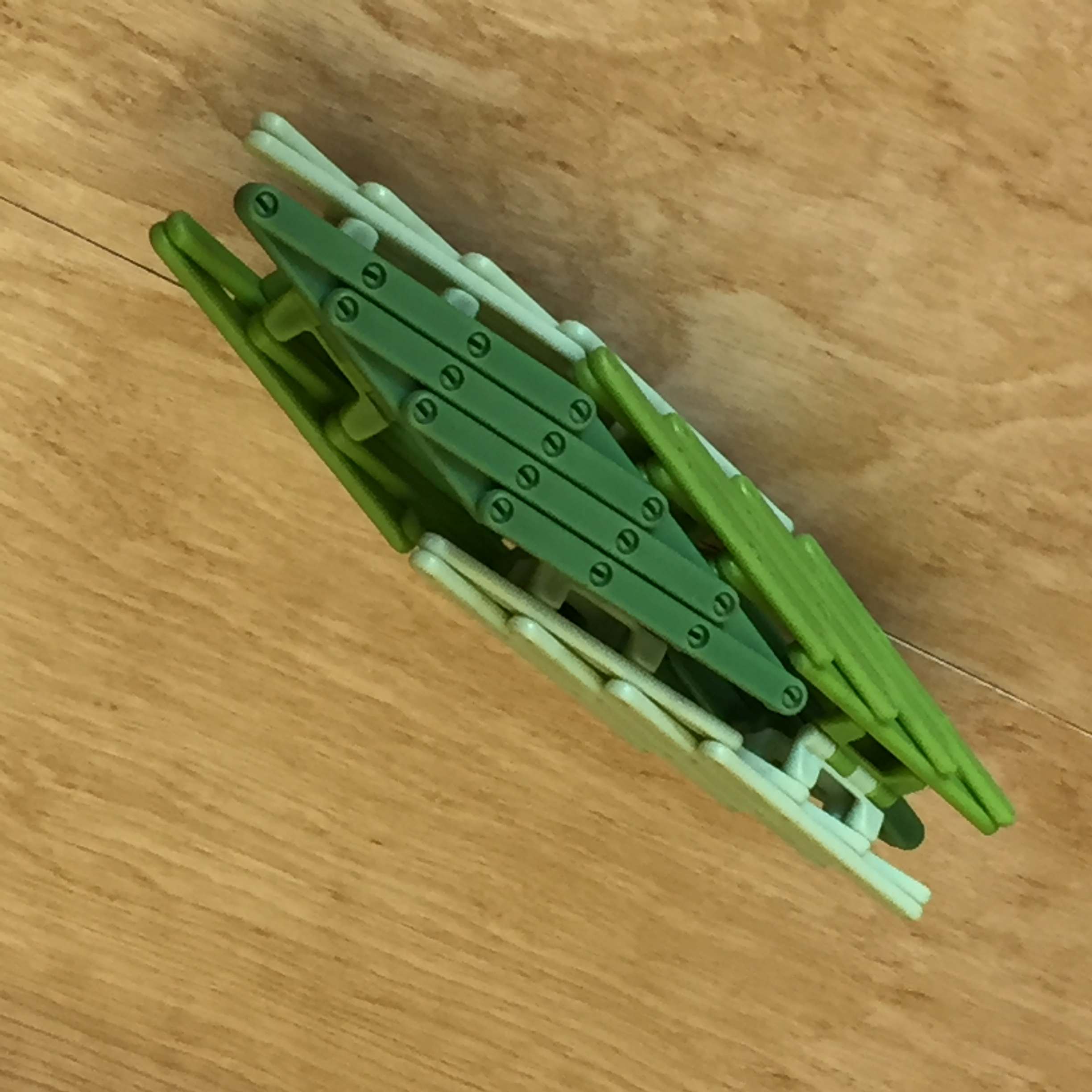}%
\end{center}
\caption{
Two Euclidean equivalent
``frameworks'' related by
an affine flex. The frameworks are not ruled, and they are not
neigborhood affine rigid.
}
\label{fig:pic}
\end{figure}

When trying to establish  global or universal rigidity, we can
often use
a certificate, called an ``equilibrium stress matrix'',
to rule out the existence of any equivalent framework
to $(G,\p)$ \emph{except for those that arise through
$d$-dimensional affine transforms}~\cite{Connelly-energy,Connelly-global}.
Then an extra argument is needed to establish that any affine transform
that preserves the $m$ lengths of $(G,\p)$ must actually be a Euclidean
transform. This last step is equivalent to proving that the edge directions
of $(G,\p)$ do not lie on a ``conic at infinity''.

Figure~\ref{fig:pic} shows a case
where there exists an affine motion applied to a framework
that does preserve all of the ``edge'' lengths.

We will use the following conventions.
Differences between pairs of points in $\EE^d$ give vectors in
a $d$-dimensional linear space, which we will identify with
$\RR^d$. Then after fixing an origin point in $\EE^d$ we will also identify
each point with its affine coordinates in $\RR^d$.

\begin{definition}
Let $(G,\p)$ be a framework in $\EE^d$.
For any edge $\{ij\}$ in $G$, let its edge vector in $\RR^d$
be
$\e_{ij}:= \p_j - \p_i$.
We say that the \textit{edge directions
of $(G,\p)$ lie on a conic at infinity} of $\EE^d$
if there exists
a non-zero symmetric $d\times d$ matrix $\Q$ such that for all of the
edge vectors, we have $\e_{ij}^t \Q \e_{ij}=0$.

(This property does not depend on the choice affine coordinates for
$\EE^d$.)

\end{definition}
In $\EE^2$, the conic condition
means that all of the edges are in, at most, two directions.
\begin{remark}
\label{rem:term}
The terminology ``conic at infinity'' comes from the following
projective setup.  Identify $\EE^d$ with the affine patch of
$\PP^d$ that has coordinates $(\cdots : 1)$. For computations,
we give a point $\x\in \EE^d$ homogeneous coordinates $\hat \x$,
which is the vector in $\RR^{d+1}$ obtained by appending a $1$
to $\x$. With this choice, difference vectors,
such as the edge vectors
$\e_{ij}$, have homogeneous coordinates (not identically zero)
of the form
$(\cdots  : 0)$, so they correspond to the point at infinity on the line
through $\p_i$ and $\p_j$.
In this view, the edge directions of $(G,\p)$
are on a conic at infinity if the $\e_{ij}$ are contained
in a quadric that lies in the hyperplane at infinity.
\end{remark}

\begin{definition}\label{def:affine-motion}
Let $A$ be an
affine map $A$ on $\EE^d$, and
define $A(\p)$ by $A(\p)_i := A(\p_i)$.
An \textit{affine
flex} of  a framework $(G,\p)$ in $\EE^d$ is an
an affine map $A$ such that $(G,A(\p))$ is
equivalent to $(G,\p)$.  An affine flex is non-trivial
if $A$ is not a Euclidean motion.
\end{definition}
The importance of conics at infinity comes from their close
connection to affine flexes of frameworks.
\begin{proposition}[\cite{Connelly-rigidity,Connelly-global}]\label{prop:no-conic-means-gr}
The edge directions of $(G,\p)$ lie on a conic at infinity
if and only if there is a non-trivial affine flex of $(G,\p)$.
\end{proposition}
\begin{remark}\label{rem:affine-local-global}
Connelly \cite{Shaping-space} has shown that if
$(G,\p)$ has a non-trivial affine flex $A$, then there is a
continuous path  $A_t$ of non-trivial affine flexes.  Thus,
Proposition \ref{prop:no-conic-means-gr} implies that
if the edges of $(G,\p)$ lie on a conic at infinity, then
$(G,\p)$ is not even ``locally rigid'' (see \cite{Gortler-affine-rigidity} for detailed definitions).
\end{remark}

\begin{figure}[h]
\begin{center}
\includegraphics[width=.2\textwidth]{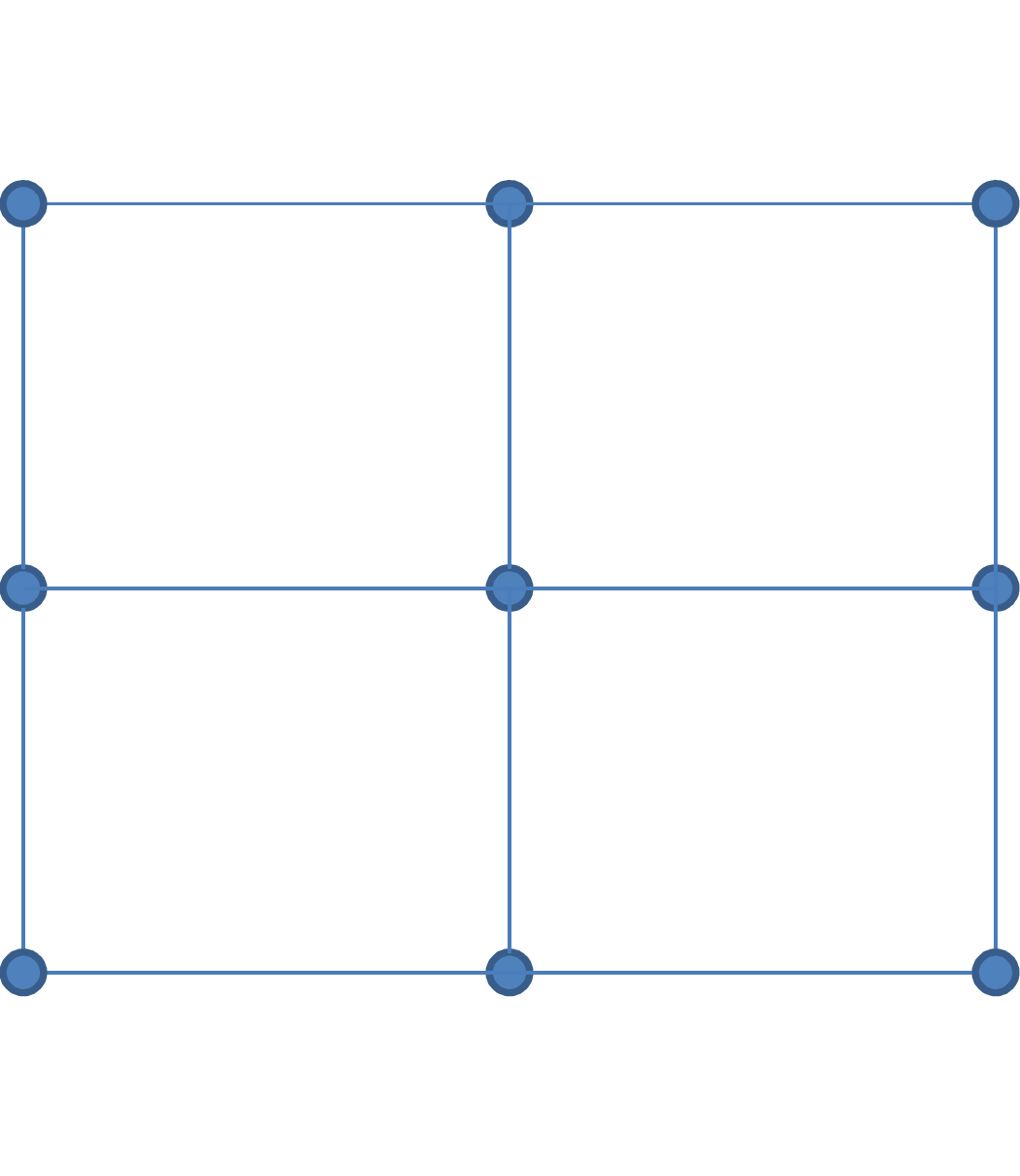}%
\hspace{1 in}
\includegraphics[width=.28\textwidth]{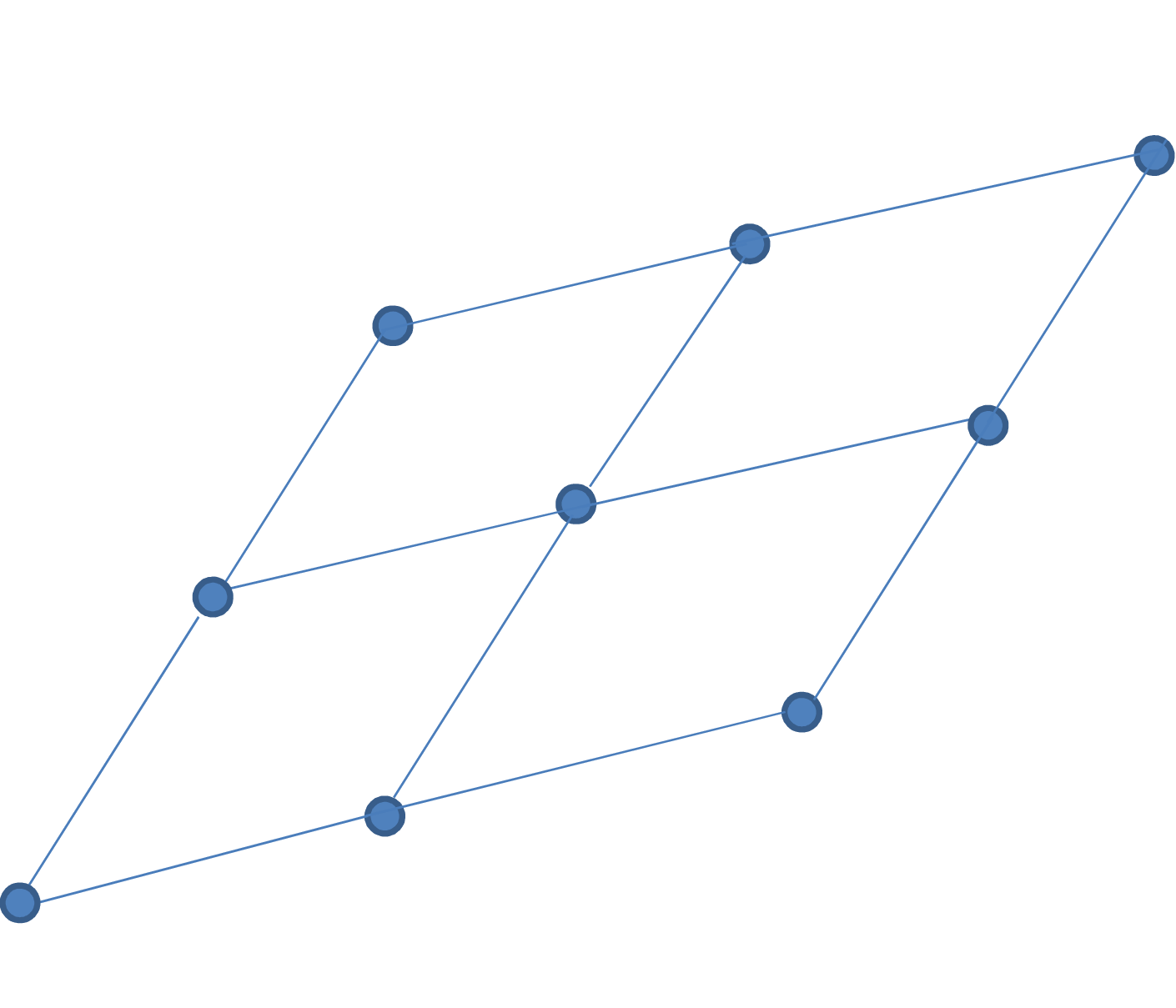}%
\end{center}
\caption{(Left) A framework $(G,\p)$ with its edge
directions at a conic at infinity.
(Right) A Euclidean equivalent
framework  obtained from $(G,\p)$ through a single
affine (but non Euclidean) transformation.
}
\label{fig:gate}
\end{figure}

Figure~\ref{fig:gate} shows a framework in $\EE^2$
with its edge directions
on a conic at infinity. A Euclidean equivalent framework on the right is obtained
obtained from $(G,\p)$ through a single
affine (but non Euclidean) transformation.

For any specific framework $(G,\p)$, we can efficiently
test whether the edge directions are on a conic at infinity
by solving a linear system. However, the usual
application of Proposition \ref{prop:no-conic-means-gr}
is to classes of frameworks.  This motivates the
question of when $(G,\p)$ does not, a priori, have
its edge directions on a conic at infinity.

A fundamental result along these lines is due to Connelly~\cite{Connelly-global}:
\begin{theorem}[\cite{Connelly-global}]\label{thm:generic-noconic}
Let  $(G,\p)$  be a framework in $\EE^d$.
Suppose that each vertex of $G$ has degree at least $d$.
Furthermore suppose that $\p$ is a generic configuration.
Then the edge directions of $(G,\p)$ do not lie
on a conic at infinity of $\EE^d$.
\end{theorem}
In Theorem \ref{thm:generic-noconic}, genericity means that the
coordinates of $\p$ do not satisfy any non-trivial algebraic equations
with coefficients in $\QQ$.  This means that, while the theorem holds
for almost every $\p$ in configuration space,
there can be exceptions.
Moreover, if one restricts oneself to some special class
of frameworks,
say, with some imposed non-trivial symmetry,
then the theorem may not hold
anywhere in that class.

Before continuing, we need two definitions:
\begin{definition}
Let $\SSS^n$ be the set of symmetric $n\times n$ matrices.
Let $C(G)$, the \emph{graph supported matrices},  be the linear space
of matrices (of any rank) in $\SSS^n$,
that vanish on the $\{ij\}$ entries corresponding
to non-edges of $G$.
\end{definition}

\begin{definition}
\label{def:stress}
A \emph{stress matrix} $\Omega$ is a matrix (of any rank)
in $C(G)$ with the added property that the
all-ones vector is in its kernel.
We say that a $d$-dimensional framework $(G,\p)$ is
\emph{in the kernel} of $\Omega$
if each of its d coordinate n-vectors is in the kernel.
In this case, we say that $\Omega$ is an \emph{equilibrium stress matrix}
for $(G,\p)$.

(The property of $\Omega$ being an equilibrium stress matrix for
$(G,\p)$
does not depend on the choice affine coordinates for
$\EE^d$.)
\end{definition}
In this paper, we are going to consider frameworks
$(G,\p)$ that are ``neighborhood affine rigid'' or
have an equilibrium stress matrix
of rank $n-d-1$.  This additional assumption (which is
not present in Theorem \ref{thm:generic-noconic})
will allow us to prove essentially tight results
about which frameworks have their edge directions on a conic
at infinity.

An earlier result using a similar setup is by
Alfakih \cite{alfakih2013affine}.
\begin{theorem}[\cite{alfakih2013affine}]
\label{thm:alf}
Suppose that a framework $(G,\p)$ in $\EE^d$
has an equilibrium stress matrix $\Omega$ of rank $n-d-1$.
Moreover suppose that each (inclusive) neighborhood of each vertex has a full $d$-dimensional affine span.
Then its edge directions do not lie on a conic at infinity of $\EE^d$.
\end{theorem}
The exact statement from \cite{alfakih2013affine}
also requires that the stress matrix is PSD.  However,
the proof, which is reliant on a
long series of linear-algebraic manipulations,
only uses the rank of the stress.

Our main result is  strictly stronger than
Theorem~\ref{thm:alf}.  The technique is also
more conceptual.  We relate a notion (described in
Section \ref{sec:affine}) called neighborhood affine
rigidity to that of edge directions not lying on a conic at infinity.
With this perspective,
a direct and simple  constructive argument can be
used to establish our theorem.  The statement
of our results needs a definition.

\begin{definition}
We say that a framework $(G,\p)$ in $\EE^d$
is \emph{ruled on a single quadric}
if all the vertices $\p_i$, and all of points on all
of the edges of the framework,
lie on some non-trivial, but possibly degenerate,
(possibly) inhomogeneous quadric $\QQQ$
in $\EE^d$. We can assume that, like $\p$, $\QQQ$ has a full
dimensional affine span.
For brevity, we will simply refer to this property as
\emph{ruled}.
(This property does not depend on the choice affine coordinates for
$\EE^d$.)

We may describe
$\QQQ$ by a defining polynomial $Q(\x) = \x^t\Q\x + \l^t\x+c$
where $\Q$ is a symmetric $d\times d$ matrix,
$\l$ is some vector in $\RR^d$ and $c$ is some constant.
Alternatively we may describe $Q(\x)$ in homogeneous coordinates as
$\hat{\x}^t \hat{\Q} \hat{\x}=0$ where $\hat{\Q}$ is
a
symmetric $(d+1)\times (d+1)$ matrix.
\end{definition}

A ruled framework is quite special. Indeed, assuming that $G$ is a connected
graph, a ruled framework in $\EE^2$ must be entirely contained within
two intersecting lines!
(See Figure~\ref{fig:ex} below.)

The main theorem of this paper is:
\begin{theorem}
\label{thm:main2}
Suppose that a
framework $(G,\p)$ in $\EE^d$
is neighborhood affine rigid.
Then its edge directions lie
on a conic at infinity of $\EE^d$ iff $(G,\p)$ is ruled.
\end{theorem}

\begin{corollary}
\label{cor:main}
Suppose that a framework $(G,\p)$ in $\EE^d$
has an equilibrium stress matrix $\Omega$ of rank $n-d-1$.
Then its edge directions lie
on a conic at infinity of $\EE^d$ iff $(G,\p)$ is ruled.
\end{corollary}
\begin{proof}
Proposition 5.9 of~\cite{Gortler-affine-rigidity} explains that a
equilibrium stress matrix of rank $n-d-1$ is a sufficient (but not
necessary) certificate of neighborhood affine rigidity.
\end{proof}

Proposition~\ref{prop:manycones2}, below, tells us
that
any framework with even a subset of $d$ vertices in general affine
position,
each with a neighborhood of full affine span, cannot be ruled.
This makes
Corollary~\ref{cor:main} stronger than Theorem~\ref{thm:alf}.

Using our theorem, we will also show as a corollary, that the property
of a framework being ``super stable'' is preserved by invertible
projective transforms of $\EE^d$ as well as the ``coning'' and
``slicing'' operations.

Finally, we will describe a
relationship  between the notion of a ruled framework to the notion
of a matrix having ``the Strong Arnold Property'' and a  related connection
that has recently been made in another paper by Alfakih~\cite{spec}.

\begin{figure}[h]
\begin{center}
\includegraphics[width=.3\textwidth]{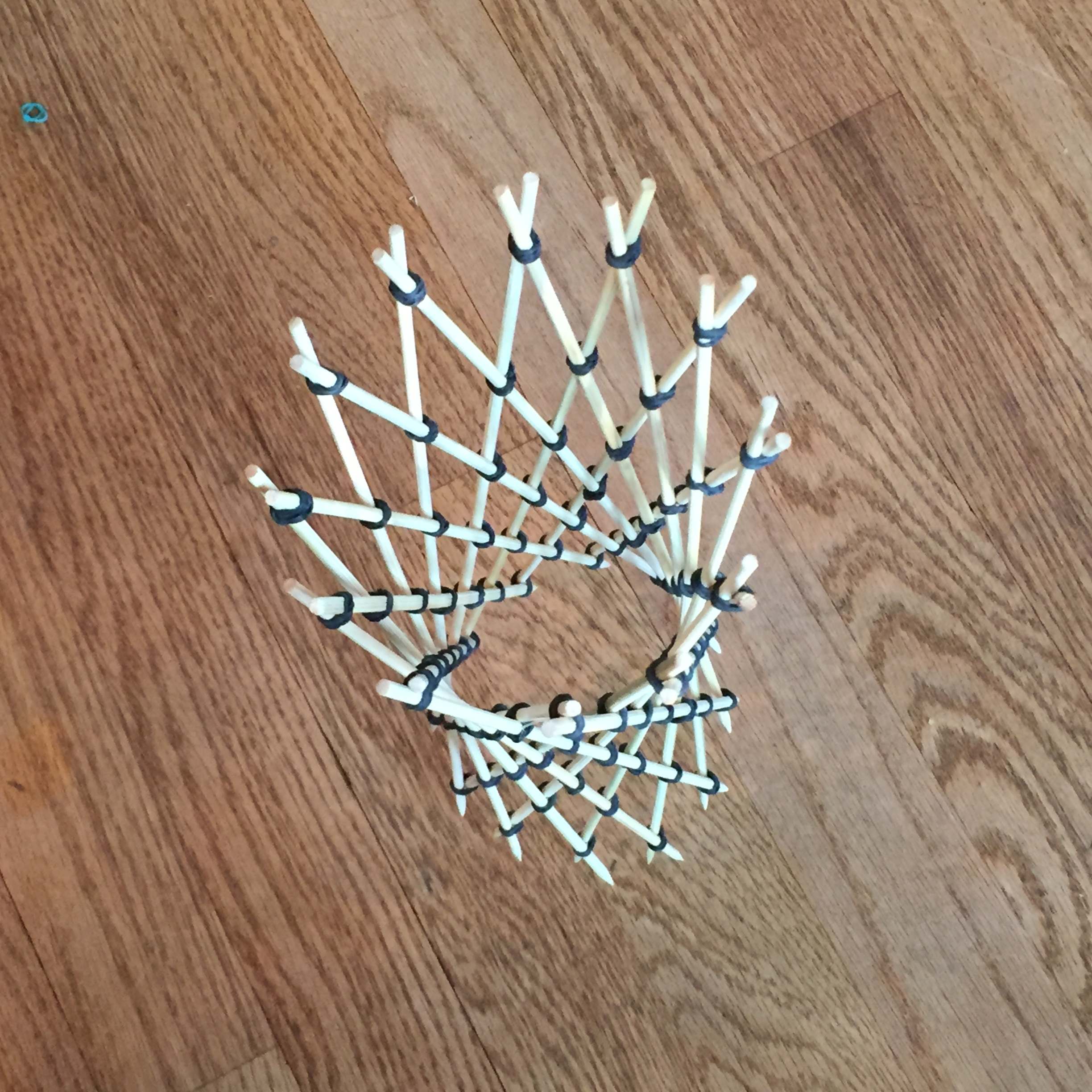}%
\hspace{1 in}
\includegraphics[width=.3\textwidth]{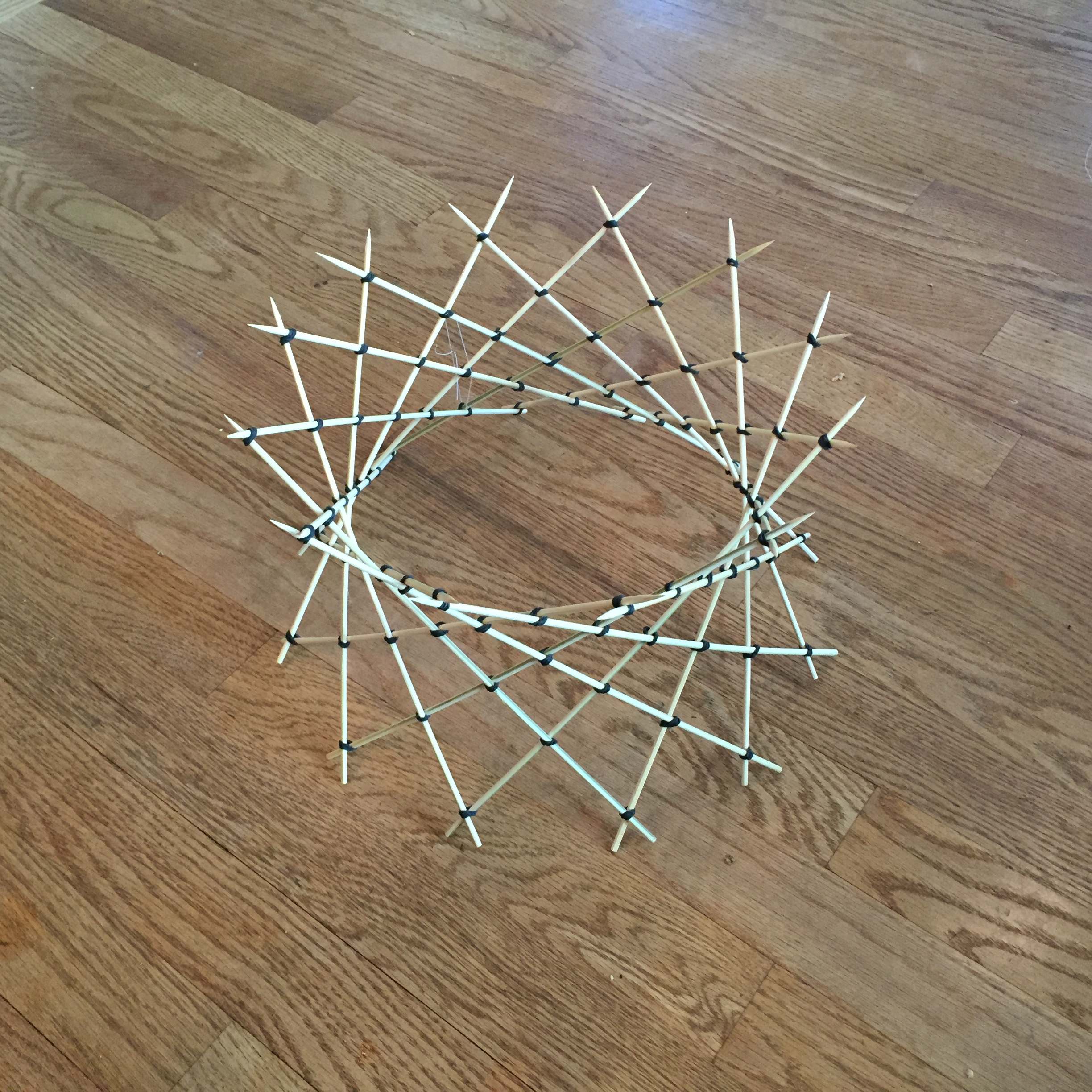}%
\end{center}
\caption{
Two Euclidean equivalent
``frameworks'' related by
an affine flex. The frameworks are
neigborhood affine rigid
and are ruled.
}
\label{fig:pic2}
\end{figure}

\section{Neighborhood Affine Rigidity}\label{sec:affine}

First  we review a few definitions about affine rigidity
from~\cite{Gortler-affine-rigidity}.

\begin{definition}
Let $(G,\p)$ be a framework in $\EE^d$.
We say that $(G,\p)$ is neigborhood
affine preequivalent to a second framework
$(G,\q)$ if for each vertex $i$, the point $\p_i$ and all of the
points $\p_j$,  where vertex $j$ is a neighbor of vertex $i$,
can be mapped to their associated positions in $\q$ by a (possibly singular)
affine transform depending only on $i$.

We say that $\p$ is affine precongruent to
$\q$ if all the vertices in $\p$
can be mapped to their positions in $\q$ by a (possibly singular)
affine transform.

(The inclusion of singular transforms is done for technical
reasons~\cite{Gortler-affine-rigidity}. This prevents  affine preequivalence
(and also precongruence) from being a symmetric relation, and is the source
for the ``pre'' terminology.)

We say that $(G,\p)$ is \emph{neighborhood
affine rigid} if for any other framework
$(G,\q)$, to which $(G,\p)$ is neighborhood
affine prequivalent, we always have that
$\p$ is affine precongruent to $\q$.
\end{definition}

\begin{definition}
Let $\Q$ be a symmetric $d\times d$ matrix. We define
an associated \emph{perturbation map} $m$ acting on a point $\x$ in
$\EE^d$ to be
\ba
m(\x) := \x + [\x^t \Q \x]\v
\ea
where
$\v$
is some chosen non-zero vector in $\RR^d$.

We denote by $m(\p)$ the configuration defined by mapping all of the
points of $\p$ by $m$.
\end{definition}

\begin{proposition}
\label{prop:ae}
Suppose $(G,\p)$ has its edge directions on a conic at infinity defined by
a non-zero matrix $\Q$. Let $m$ be an associated perturbation map.
Then $(G,\p)$ is neighborhood affine preequivalent to $(G,m(\p))$.
\end{proposition}
\begin{proof}
Let $\q:=m(\p)$.
We just need to show that
for each vertex $i$, there is an affine transform that maps
each vertex $\p_j$ of the (inclusive) neighborhood of $\p_i$,
to its position in
the configuration $\q$.

From our assumption that the edge directions are at a conic at infinity,
\ba
0 &=& (\p_j-\p_i)^t \Q (\p_j-\p_i)
\ea
we get
\ba
\p_j^t \Q \p_j
&=&
-\p_i^t \Q \p_i +2 \p_i^t \Q \p_j
\ea
when vertex $j$ is a neighbor of vertex $i$.
The same is trivially true when $j = i$.

Thus
on the (inclusive) neighborhood of $\p_i$,
the action of $m$
can
be modeled with the affine transform:
\ba
m(\x) = \x + [
-\p_i^t \Q \p_i +2 \p_i^t \Q \x
]\v
\ea

\end{proof}

\begin{lemma}
\label{lem:ponq}
Let $\Q$ be a non-zero symmetric $d\times d$ matrix and
let $m$ be an associated perturbation map.
Suppose that $\p$ is  affine precongruent to $m(\p)$.
Then all of the $\p_i$ must lie on a (possibly) inhomogeneous quadric
with its quadratic term defined by $\Q$.
\end{lemma}
\begin{proof}
Assume, w.l.o.g., that $\|v\|=1$. Expanding the
definition of $m$ and rearranging, affine precongruence
implies there is an affine map $A(\x) := \A\x + \t$
(where $\A$ is a $d\times d$ matrix, and $\t$
is a vector in $\RR^d$)
such that for all $\x=\p_i$
\[
m(\x) = \x + [\x^t \Q \x]\v = \A\x + \t
\]
and
\[
[\x^t \Q \x]\v = (\A-\I)\x + \t =:\A'\x + \t
\]
Multiplying by $\v^t$ on the left, we obtain
\[
\x^t \Q \x = \left[\v^t\A'\right]\x + \v^t\t
\]
which gives a non-trivial quadric for the $\p_i$, since the left-hand side, at
least, is non-zero.
\end{proof}

\begin{lemma}
\label{lem:line}
Let $\QQQ$ be a
(possibly) inhomogeneous quadric with quadratic terms defined by
a non-zero symmetric
matrix $\Q$. Suppose that two points $\x_1$ and $\x_2$ are both on $\QQQ$,
and that for  the edge vector $\e:=\x_2-\x_1$ we have $\e^t\Q\e=0$.
Then all the points on the line spanned by $\x_1$ and $\x_2$ are on $\QQQ$.
\end{lemma}
\begin{proof}
Thinking projectively (see Remark~\ref{rem:term}),
we have the points $\x_1$, $\x_2$ and $\e$ (at infinity)
on a line $\ell$ in $\PP^d$.  If $Q$ is the polynomial defining
$\QQQ$, then $Q(\e) = \e^t\Q\e$, since $\e$ is at infinity.
Thus $\QQQ$ has more than $2$ intersection points with
$\ell$, so $Q$ vanishes identically on it.
\end{proof}
\begin{remark}\label{rem:line-affine}
We could prove Lemma \ref{lem:line} without
leaving the affine setting.  Suppose that
$\QQQ$ is defined by the polynomial $Q(\x) = \x^t\Q\x + \l^t\x+c$.
The hypothesis about the edge vector implies the identity
$\x^t_1\Q\x_1 + \x^t_2\Q\x_2 = 2\x_1^t \Q \x_2$.  We then
compute that $Q(\frac{1}{2}[\x_1+\x_2]) = \frac{1}{2}[Q(\x_1)+Q(\x_2)] = 0$,
which lets us proceed as above.
\end{remark}
We will see a different approach to Lemma \ref{lem:line} in the
discussion about the Strong Arnold Property in Section
\ref{sec:arnold}, below.

\begin{proposition}
\label{prop:pruled}
Suppose $(G,\p)$
has its edge directions on a conic at infinity defined by
$\Q$ and
let $m$ be an associated perturbation map.
Suppose that $\p$ is  affine precongruent to $m(\p)$.
Then $(G,\p)$ must be a ruled framework.
\end{proposition}
\begin{proof}
From Lemma~\ref{lem:ponq}, we have each point $\p_i$
on a (possibly) inhomogeneous quadric, with its quadratic terms defined by a $d\times d$ matrix
$\Q$.
By assumption,
for each edge vector, $\e_{ij}:=\p_j-\p_i$, we have $\e_{ij}^t\Q\e_{ij}=0$.
Thus from Lemma~\ref{lem:line} we see that $(G,\p)$ must be ruled.
\end{proof}

We are now ready to prove our main result.
\begin{proof}[Proof of Theorem \ref{thm:main2}]
Let $(G,\p)$ be a framework that is neighborhood affine rigid.  If
$(G,\p)$ is ruled on a quadric $\QQQ$, the points
of $\QQQ$ at infinity are a conic at infinity.
As in the proof of Lemma \ref{lem:line},
the edge direction vectors are on this conic.

For the ``only if'' direction,  assume that
it has its edge directions on a conic
at infinity. Let $m$ be an associated perturbation map.
Then from Proposition~\ref{prop:ae}
$(G,\p)$ is neighborhood affine preequivalent to $(G,m(\p))$.
Since $(G,\p)$ is neighborhood affine rigid, this implies that
$\p$ is affine precongruent to $m(\p)$.
So from Proposition~\ref{prop:pruled}, $(G,\p)$ must be ruled.
\end{proof}

\begin{figure}[h]
\begin{center}
\includegraphics[width=.2\textwidth]{pee}%
\hspace{1 in}
\includegraphics[width=.2\textwidth]{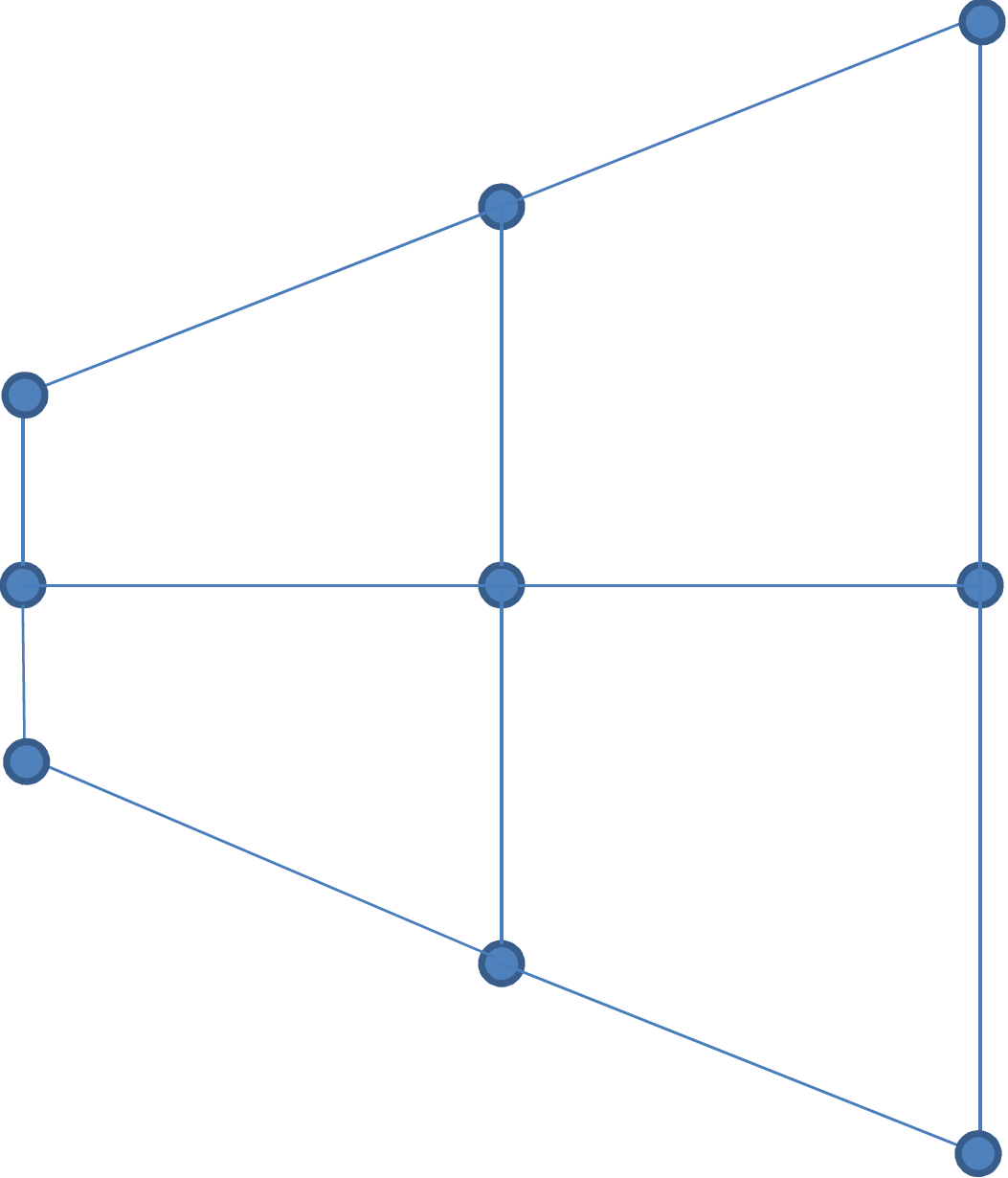}%
\end{center}
\caption{(Left) A framework $(G,\p)$ with its edge directions
at a conic at infinity.
(Right) A framework $(G,\q)$ with the property that each vertex neighborhood
in $(G,\q)$ can be obtained from its corresponding neighborhood in $(G,\p)$
through an affine transform. There is no global affine transformation that
maps $\p$ to $\q$.
}
\label{fig:grid}
\end{figure}

Figure~\ref{fig:grid} shows a simple example of the construction used in
the proof of Theorem~\ref{thm:main2}. In the original framework
$(G,\p)$, shown
on the left,
all of the edge directions are either horizontal or vertical. Thus,
they lie on the conic at infinity defined by the equation $xy=0$.
However, this framework is not ruled.
We will see that it is not neighborhood affine rigid.

To define the perturbation map,
we set
the origin point to be the center vertex of the framework, which,
for simplicity,
we will assume
has Cartesian coordinates $[0,0]^t$ and
we set
the vector
$\v$ to be $[0,1]^t$ (in the vertical direction).
The resulting
map can be
described, in coordinates, by
\ba
[x,y]^t \mapsto [x,y+xy]^t
\ea
On the right we show  $(G,\q)$,
the image of $(G,\p)$ under this map.
It is easy to see  that each vertex neighborhood
of $(G,\q)$ can be obtained from its configuration in $(G,\p)$ under
some affine transform, but that the full configuration $\q$
cannot be obtained from $\p$ using a single, global affine
transform.

\section{What do ruled frameworks look like?}
We now unpack what it means to have a ruled framework.

\begin{definition}
Let $S$
be a point set in $\EE^d$.
We say that $\x$ is a \emph{cone point} of $S$
if for any other point $\y$ on $S$, we have the entire line spanned by
$\x$ and $\y$ is in $S$.
\end{definition}

The following is standard
(see e.g. \cite[Example 3.3 and Lecture 22]{harris2013algebraic}).
Every non-trivial
quadric $\QQQ$
can be put in canonical form under a projective transform
by diagonalizing and normalizing $\hat{\Q}$, the
$(d+1)\times (d+1)$ symmetric matrix
that describes $\QQQ$
in homogenized form.
The resulting
canonical matrix will have some number of $+1$, $-1$ and $0$ diagonal
entries.
Call the rank of this matrix $r$.
If the matrix is definite, or semi-definite, then $\QQQ$ cannot
have full affine span. So let us now assume that the matrix is indefinite.
When $r=d+1$, the quadric is smooth.
Otherwise, (we have $1<r<d+1$),
the
quadric is the cone over a smooth quadric
of dimension $r-2$ in $\EE^{r-1}$,
with a cone point set  comprising an affine space of
dimension $d-r$.
In particular, all non-smooth points must
be cone points.
(See Figure~\ref{fig:3d} below.)

This canonical picture immediately gives us the following two Lemmas.

\begin{lemma}
\label{lem:onecone}
Let $\x$ be a point of a quadric $\QQQ$ in $\EE^d$
that has $d$ linearly independent
ruling directions
(lines within the conic through that point).
Then $\x$ must be a cone point
of $\QQQ$.
\end{lemma}
\begin{proof}
The tangent hyperplane at a smooth point must include all ruling directions
through  that point.
So
having $d$ ruling directions precludes the existence of such a tangent.
\end{proof}

\begin{lemma}
\label{lem:manycones}
Suppose $\QQQ$ is a non-trivial
(possibly) inhomogeneous quadric with a full affine
span in
$\EE^d$. Then $\QQQ$ cannot have $d$ cone points in general affine position.
\end{lemma}
\begin{proof}
The full affine span means that the rank of the quadric must be greater than
$1$ and indefinite.
Thus, from our canonical picture of quadrics,
the dimension of the affine space of cone points must be
no larger than $d-2$.
\end{proof}

\begin{proposition}
\label{prop:manycones2}
Suppose $(G,\p)$, a framework in $\EE^d$,
has a subset of $d$ vertices in general affine position,
each with neighborhood in $(G,\p)$ with full affine span.
Then $(G,\p)$ cannot be ruled.
\end{proposition}
\begin{proof}
By assumption, $(G,\p)$ has a full affine span.
Suppose that $(G,\p)$ were ruled by a non-trivial quadric $\QQQ$.
Then from Lemma~\ref{lem:onecone},
each of the $d$ vertices would be a cone point.
But this would contradict Lemma~\ref{lem:manycones}.
\end{proof}

\begin{remark}
\label{rem:stronger}
We can now see that Corollary~\ref{cor:main} is stronger than
Theorem~\ref{thm:alf}.
In light of Proposition~\ref{prop:manycones2}
the assumed affine span condition
of Theorem~\ref{thm:alf}
is strictly stronger than being non-ruled.
\end{remark}

\section{Super Stability}

Recall that a framework in $\EE^d$ is universally rigid if
there is no  second framework $(G,\q)$ in any dimension
that is equivalent but not congruent to $(G,\p)$.
Universal rigidity (unlike infinitesimal rigidity) is
known not to be invariant under projective transformations
or under the
``slicing'' operation described
below~\cite{Connelly-Gortler}.
A slightly stronger property than universal rigidity is
called super stability~\cite{Connelly-energy,Gortler-thurston2}.
In this section, we show that super stability is well-behaved with
respect to these operations.  We start with some definitions.

\begin{definition}
A framework $(G,\p)$ in $\EE^d$ with a full dimensional affine span is called
\emph{super stable} if it has a positive semidefinite
(PSD) equilibrium stress matrix
$\Omega$ of rank $n-d-1$, and its edge directions do not lie on a conic at
infinity.
\end{definition}

\begin{theorem}
[\cite{Connelly-energy}]
Super stability implies universal rigidity.
\end{theorem}

\begin{definition}
A \emph{cone graph} is a graph with $n+1$ vertices
where vertex $0$
is connected to
all the others.
We will consider \emph{cone frameworks} of this graph
in $\EE^{d+1}$
denoted
as $\p_0*(G,\p)$.
Here $G$ is the subgraph on $n$ vertices
induced by removing vertex $0$,
$\p$ is
a configuration of the $n$ vertices of $G$ in $\EE^{d+1}$,
and
$\p_0$ is another point in $\EE^{d+1}$ that is not coincident
with any of the
points of $\p$.
\end{definition}
We define several operations on cone frameworks.
\begin{definition}
Given a cone framework $\p_0*(G,\p)$ in $\EE^{d+1}$, the process of
\emph{sliding} denotes
moving points of $\p$ along their lines
connecting them to $\p_0$, while avoiding $\p_0$ itself.
\end{definition}
\begin{definition}
We call a cone framework
$\p_0*(G,\p)$ in $\EE^{d+1}$,
\emph{flat} if each of the $n$ vertices
of $\p$ lies in a $d$-dimensional Euclidean subspace that does not
include $\p_0$.
\end{definition}
\begin{definition}
Given a framework $(G,\p)$ in $\EE^d$, we can \emph{cone} it
by placing $\EE^d$ in a hyperplane in
$\EE^{d+1}$, and then adding a cone vertex
at some location $\p_0$ (outside of the affine span of $\p$)
to create a flat cone framework in
$\EE^{d+1}$.

Given a cone framework $\p_0*(G,\p)$ in $\EE^{d+1}$, we can
\emph{slice} it, by sliding all points of $\p$
to create a flat cone framework.
And then we can consider the resulting subframework of $G$
as living in $\EE^d$.
\end{definition}
This section's main result is:
\begin{theorem}\label{thm:coneslice}
Super stability is invariant under coning and slicing.
\end{theorem}
\begin{corollary}\label{cor:ssinv}
Super stability is invariant with respect to invertible
projective transformations in $\EE^d$ that do not
send any vertices to infinity.
\end{corollary}
\begin{proof}
Any projective transformation on $(G,\p)$ can be modeled
by coning, followed by a linear transformation on $\EE^{d+1}$,
followed by slicing.
\end{proof}
\subsection{Proof of Theorem \ref{thm:coneslice}}
We start with two technical lemmas.
\begin{lemma}\label{lem:slide}
Suppose that a cone framework $\p_0*(G,\p)$ has an equilibrium
stress matrix
$\Omega$.  Then any framework $\p_0*(G,\q)$ obtained by
sliding has an equilibrium
stress matrix $\Omega'$ of the same rank and
signature.
\end{lemma}
The proof is based on ideas of~\cite[Theorem 8]{lovasz2001steinitz}.
See also~\cite[Lemma 4.11]{laurent2014positive}.
\begin{proof}
By translating, we can assume that $\p_0$ is at the origin,
and that sliding is then represented by  scaling:
$\q_i := s_i \p_i$.

Let $\Omega$ be an equilibrium stress matrix for
$\p_0*(G,\p)$
with signature $(a,b,c)$ ($a$ negative eigenvalues,
$b$ zero eigenvalues and $c$ positive eigenvalues).
Let $\Psi$ be the matrix obtained from $\Omega$ by removing
the row and column corresponding to the cone vertex.
Let it have signature $(f,e,g)$.
Because the kernel of
$\Omega$ contains a vector that is non-zero on the coordinate
corresponding to the cone vertex
(such as
the all-ones vector), we know that
the rank of $\Psi$ is the same as the rank of $\Omega$,
thus it
has one less zero eigenvalue. Then from the eigenvalue interleaving
theorem, we must have $f=a$ and $g=c$.

Since $\p_0$ is at the origin, we must still have the $d+1$
coordinate vectors of $\p$ in the kernel of $\Psi$, though the all-ones
vector is no longer in the kernel.

When scaling $\p$ to obtain $\q$, we can define the matrix
$\Psi' := S \Psi S$, where $S$ is a full rank diagonal matrix
with entries $1/s_i$. The matrix $\Psi'$ will have the same
signature as $\Psi$ and will have the coordinates of $\q$ in
its kernel.

Finally we will augment $\Psi'$ with a row and column corresponding
to the cone vertex so that the all-ones vector is in the kernel of the
resulting stress $\Omega'$. To do this, we can first add a column
which is the negative sum of the $n$ columns of $\Psi'$.
Since $\p_0$ is at the origin, the coordinates of
$\p_0*(G,\q)$ are in the kernel of this matrix.
Then we can add a row which is the negative sum of the $n$ rows
of this intermediate matrix.
The coordinates of
$\p_0*(G,\q)$ must also be annihilated by this last row
(as it is just a linear combination of the other rows).
Thus we have an equilibrium stress matrix for
$\p_0*(G,\q)$.
The rank of $\Omega'$ is the same as $\Psi'$
thus it
has one more zero eigenvalue. Again,  from the eigenvalue interleaving
theorem, it must have signature $(a,b,c)$.
\end{proof}

\begin{lemma}
\label{lem:conar}
A cone framework $\p_0*(G,\p)$ has its edge directions on a conic at infinity
if and only if it is ruled.
\end{lemma}
\begin{proof}
Since $\p_0$ is connected to all other vertices, this
automatically makes the framework neighborhood affine rigid.
Then we can apply Theorem~\ref{thm:main2}.
\end{proof}
We can now prove an intermediate result that is interesting in
its own right.
\begin{proposition}\label{prop:slideinv}
Super stability of a cone framework $\p_0*(G,\p)$ is invariant with respect to
sliding.
\end{proposition}
\begin{proof}
If $\p_0*(G,\p)$ has its edge directions on a conic at infinity, then
from Lemma~\ref{lem:conar},
it is
ruled on a quadric $\QQQ$.  For each edge $\{i,j\}\in E$, the three
edges of
the triangle
$\{\p_0, \p_i,\p_j\}$ are contained in $\QQQ$, so its entire supporting
plane must be too.
So any
$\p_0*(G,\q)$ obtained by sliding is also ruled.  Thus, for cone
frameworks, having edge directions
on a conic at infinity is invariant with respect to sliding.
Lemma \ref{lem:slide} says that having a PSD equilibrium
stress matrix of rank
$n - d -1$ is as well.
\end{proof}

We can now complete the proof of Theorem \ref{thm:coneslice}.
The main
observation is that
for both $(G,\p)$ in $\EE^d$ and
$\p_0*(G,\p)$ in $\EE^{d+1}$, the necessary rank for
super stability is $n-d-1$.
Starting with an equilibrium stress for $(G,\p)$ in $\EE^d$,
we can simply add a row and column of zeros and obtain
an
equilibrium stress matrix for  its coned framework  in $\EE^{d+1}$
of the same rank and positive/negative signature.

Conversely,
for a flat cone framework $\p_0*(G,\p)$ in $\EE^{d+1}$,
any equilibrium
stress $\Omega$,
must have  $\Omega_{0i}=0$ for all $i$.
(The equilibrium condition can be thought of as a balance of
forces along the edges of the framework,
$\sum_{j\neq i} \Omega_{ij} (\p_i-\p_j) =\0$.
Any non-zero force along the cone edge at $\p_i$ cannot be
matched by the forces arising from edges within $G$  as these forces
all lie in a
hyperplane not containing $\p_0$).
By simply discarding the  row and column for $\p_0$,
we obtain
an
equilibrium stress matrix for  $(G,\p)$ in $\EE^d$
of the same rank and positive/negative signature.

Proceeding
in the coning (easy) direction,
if $(G,\p)$
has a
PSD equilibrium
stress matrix of maximum rank, then from the above observation,
so too must the result of coning.
Meanwhile, if the edge directions of
$(G,\p)$ are not on a conic at infinity, then
neither are the edges of the coned result.

Proceeding
in the slicing (harder) direction,
in light of
Proposition \ref{prop:slideinv},
we can start with
a flat coned framework $\p_0*(G,\p)$
that is super stable.
From the observation above, if $\p_0*(G,\p)$
has a
PSD equilibrium
stress matrix of maximum rank, then
so too must the sliced result $(G,\p)$ in $\EE^d$.

Meanwhile, if the edge directions of
$\p_0*(G,\p)$
are not at a conic at infinity,
then $\p_0*(G,\p)$ is certainly not ruled.
Importantly, this implies that $(G,\p)$ in $\EE^d$
is not ruled either.
From Corollary~\ref{cor:main} then, $(G,\p)$
does not have its edge directions
on a conic at infinity. Thus it is super stable. \qed

\subsection{Remarks}
Lemma \ref{lem:slide} is implicit in
\cite{coning}.  Also, we can prove Corollary \ref{cor:ssinv}
without Theorem \ref{thm:coneslice} using Corollary~\ref{cor:main},
the fact that being ruled is preserved under projective transforms,
and a result from \cite{coning} that the rank and signature of
stress matrices is preserved by invertible projective maps.

Another connection is to the notion of dimensional rigidity,
which was introduced by Alfakih \cite{Alfakih-dim-rigidity}.
A framework $(G,\p)$ in $\EE^d$ is \textit{dimensionally rigid}
if there are no equivalent frameworks with a higher dimensional
span. Alfakih \cite{Alfakih-fields} has shown that
$(G,\p)$ is dimensionally rigid but not universally rigid
if and only if its edge directions are on a conic at infinity.
Connelly and Gortler
\cite{Connelly-Gortler,connelly2015universal} showed that
dimensional rigidity is invariant with respect to
projective transformations and coning/sliding/slicing.

Universal rigidity is preserved under coning and sliding,
but it is neither projectively invariant nor is it
preserved by slicing, since having edge directions on a
conic at infinity isn't preserved by projective transforms
or by coning~\cite{Connelly-Gortler}.
The counter-examples are necessarily not neighborhood affine rigid,
and not ruled. In contrast, being ruled is invariant with
respect to invertible projective transforms and to coning.

\section{Strong Arnold Property}\label{sec:arnold}

In the literature on the Colin de Verdière graph parameter,
there is an central non-degeneracy
property of a matrix called the Strong Arnold Property~\cite{colin}.

\begin{definition}
Let $r$ be some rank.
Let $D_r$ be the determinantal variety of matrices in $\SSS^n$
with rank no greater than $r$.

A graph supported matrix, $\Psi \in C(G)$, with rank $r$
is said to satisfy the \emph{Strong Arnold
Property (SAP)} if $D_r$ and $C(G)$ intersect transversely
at $\Psi$.
\end{definition}

Recently, Laurent and Varvisiotis \cite{laurent2014positive}
began an exploration on the
relationship between universal rigidity, PSD matrix completion and
the Strong Arnold Property.  They quote an older result of Godsil,
which we translate into our language and specialize to stress matrices.
\begin{theorem}[{\cite[Theorem 3.2]{god}}]
\label{thm:vl}
Let $\Omega$ be a stress matrix with rank $n-d-1$. Let
$(G,\p)$ be a framework in its kernel with a $d$-dimensional affine span.
Then $\Omega$ does not have the SAP if and only if  $(G,\p)$ is  ruled.
\end{theorem}
\begin{remark}
Godsil's proof of Theorem \ref{thm:vl} provides another approach to
proving Lemma \ref{lem:line}.  He shows that the conditions
for a non-transverse intersection at $\Omega$ are equivalent
to the existence of a non-zero
symmetric $(d+1)\times (d+1)$ matrix $\hat{\Q}$ such that
for all vertices $i$,
we have $\hat{\p}_i^t \hat{\Q} \hat{\p}_i=0$ and that for all
edges $\{ij\}$, we have $\hat{\p}_i^t \hat{\Q} \hat{\p}_j=0$
(where $\hat{\p_i}$ are as in Remark \ref{rem:term}).
Then
$\hat{\Q}$ is the matrix defining the ruling quadric $\QQQ$.
\end{remark}
Other closely related results can be found in \cite{van2002graphs,laurent2014positive}.

Meanwhile, Alfakih proves the following:
\begin{theorem}[{\cite[Corollary 2]{spec}}]
\label{thm:alf2}
Let $\Omega$ be a stress matrix with rank $n-d-1$. Let
$(G,\p)$ be a framework in its kernel with a $d$-dimensional affine span.
Then $\Omega$ does not have the SAP
if and only if the edge directions of  $(G,\p)$ are on a conic at infinity.
\end{theorem}
\begin{remark}
Similarly to the proof of Theorem~\ref{thm:vl},
Alfakih's proof of Theorem \ref{thm:alf2} also (implicitly) involves
the existence of a non-zero
symmetric $(d+1)\times (d+1)$ matrix $\hat{\Q}$ such that
for all $i$, we have $\hat{\p}_i^t \hat{\Q} \hat{\p}_i=0$ and that for all
edges $\{ij\}$ we have
$\hat{\p}_i^t \hat{\Q} \hat{\p}_j=0$.
But
in this case, a fair amount of extra work is needed
to get to this condition starting from the assumed $\Omega$ and
assumed conic at infinity.
\end{remark}

We can  now see that Corollary~\ref{cor:main}, Theorem~\ref{thm:vl}
and Theorem~\ref{thm:alf2} form a cycle of relationships. Any two
of them imply the third.

\begin{figure}[h]
\begin{center}
\includegraphics[width=.2\textwidth]{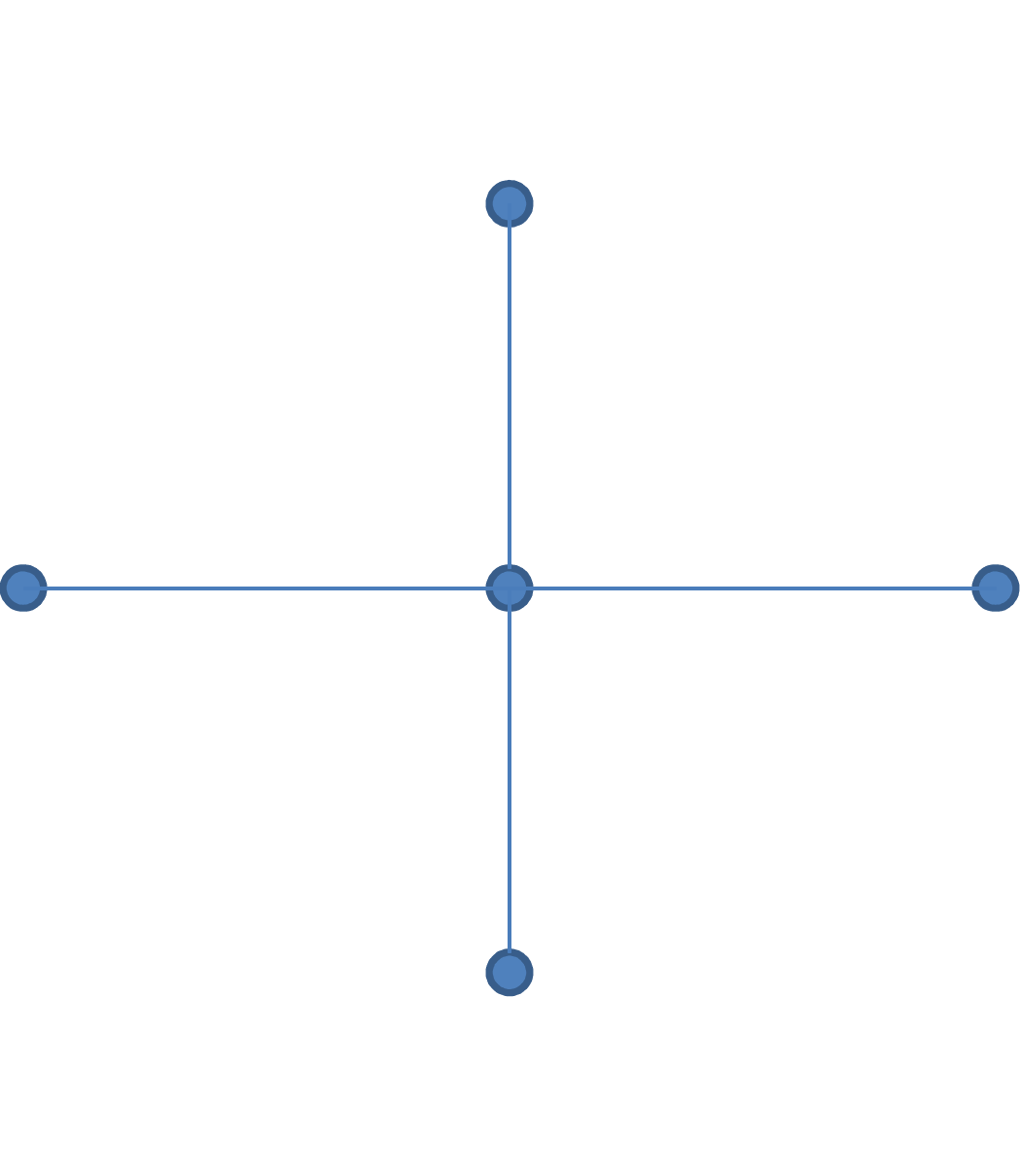}%
\hspace{1 in}
\includegraphics[width=.2\textwidth]{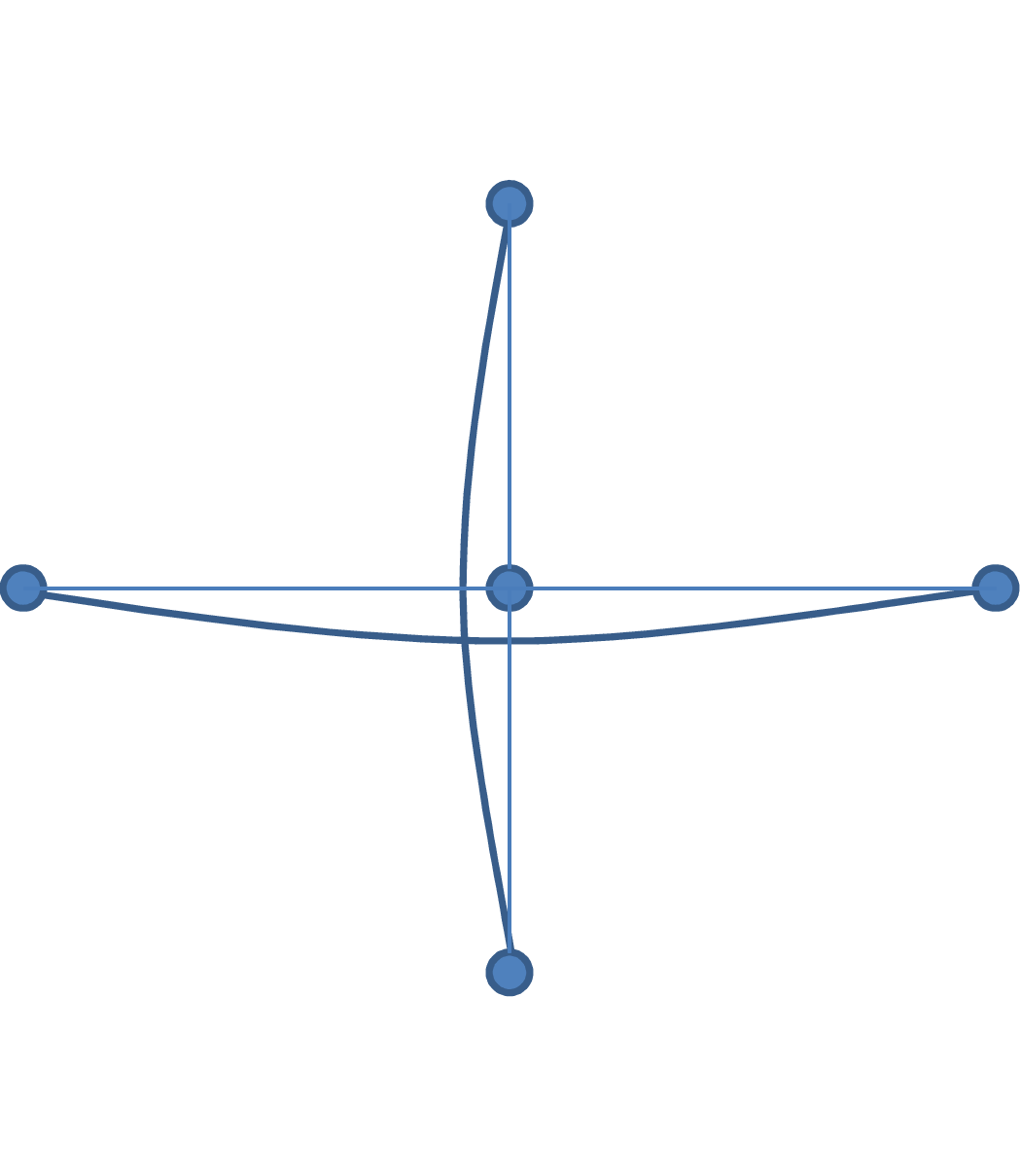}%
\end{center}
\caption{(Left) A ruled framework on two lines
that is neighborhood affine
rigid.
(Right) After adding two long bracing edges, the ruled framework
has
a PSD equilibrium stress matrix of rank $n-d-1$.
}
\label{fig:ex}
\end{figure}

\section{Examples}

\begin{figure}[h]
\begin{center}
\includegraphics[width=.15\textwidth]{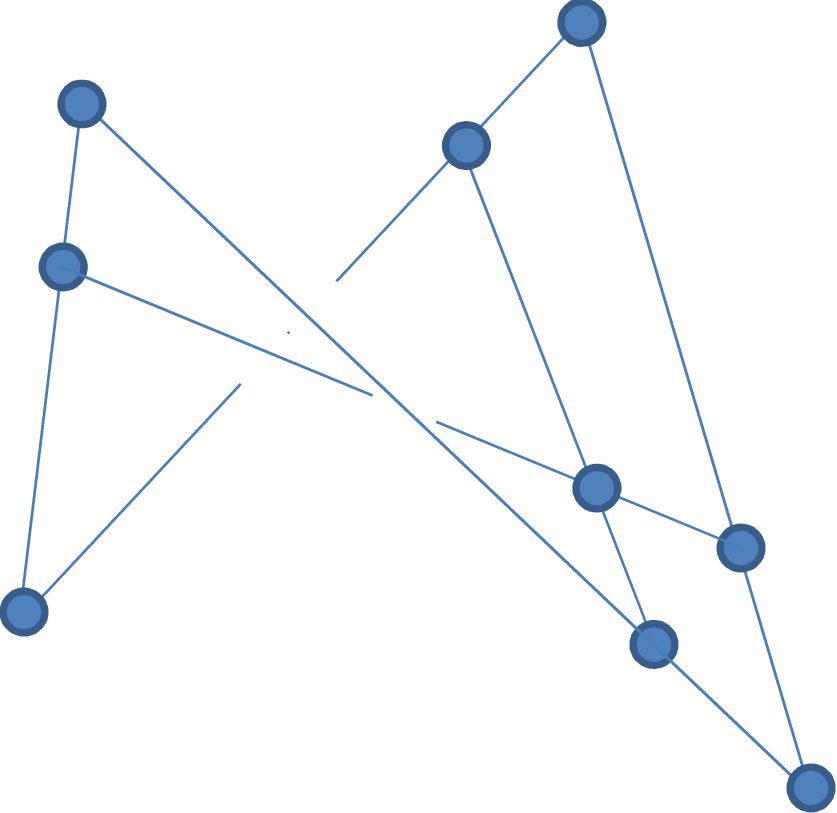}%
\hspace{.4 in}
\includegraphics[width=.2\textwidth]{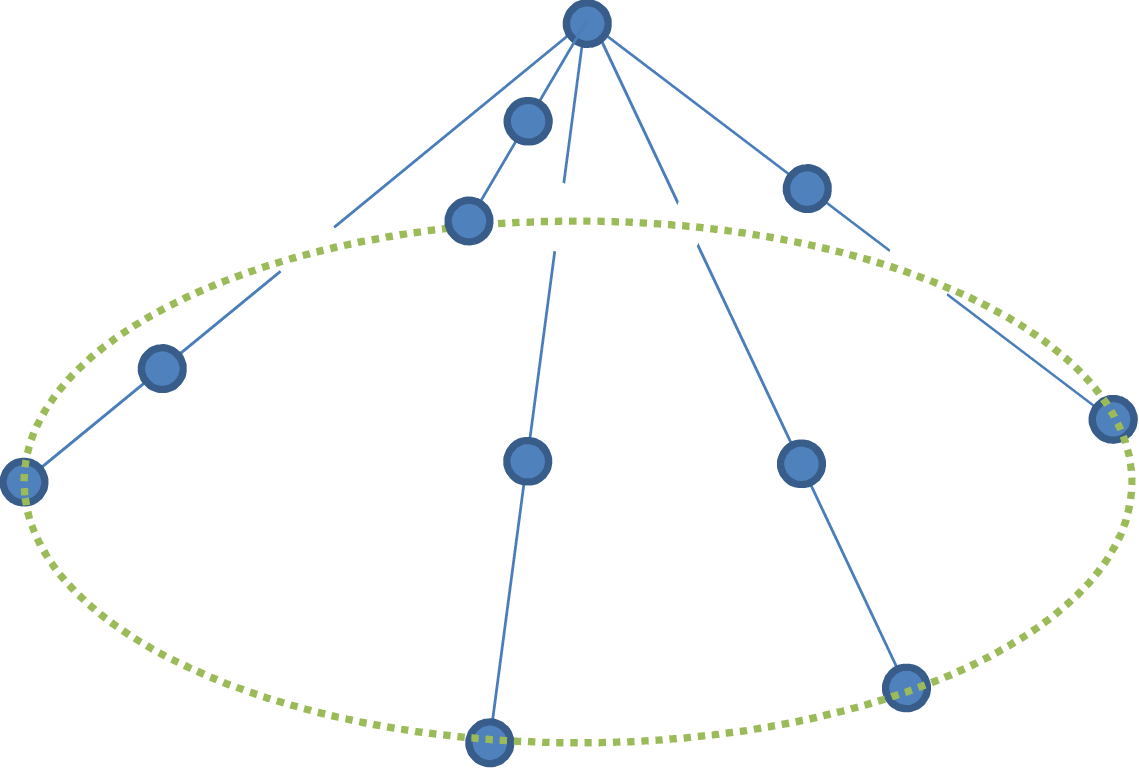}%
\hspace{.4 in}
\includegraphics[width=.2\textwidth]{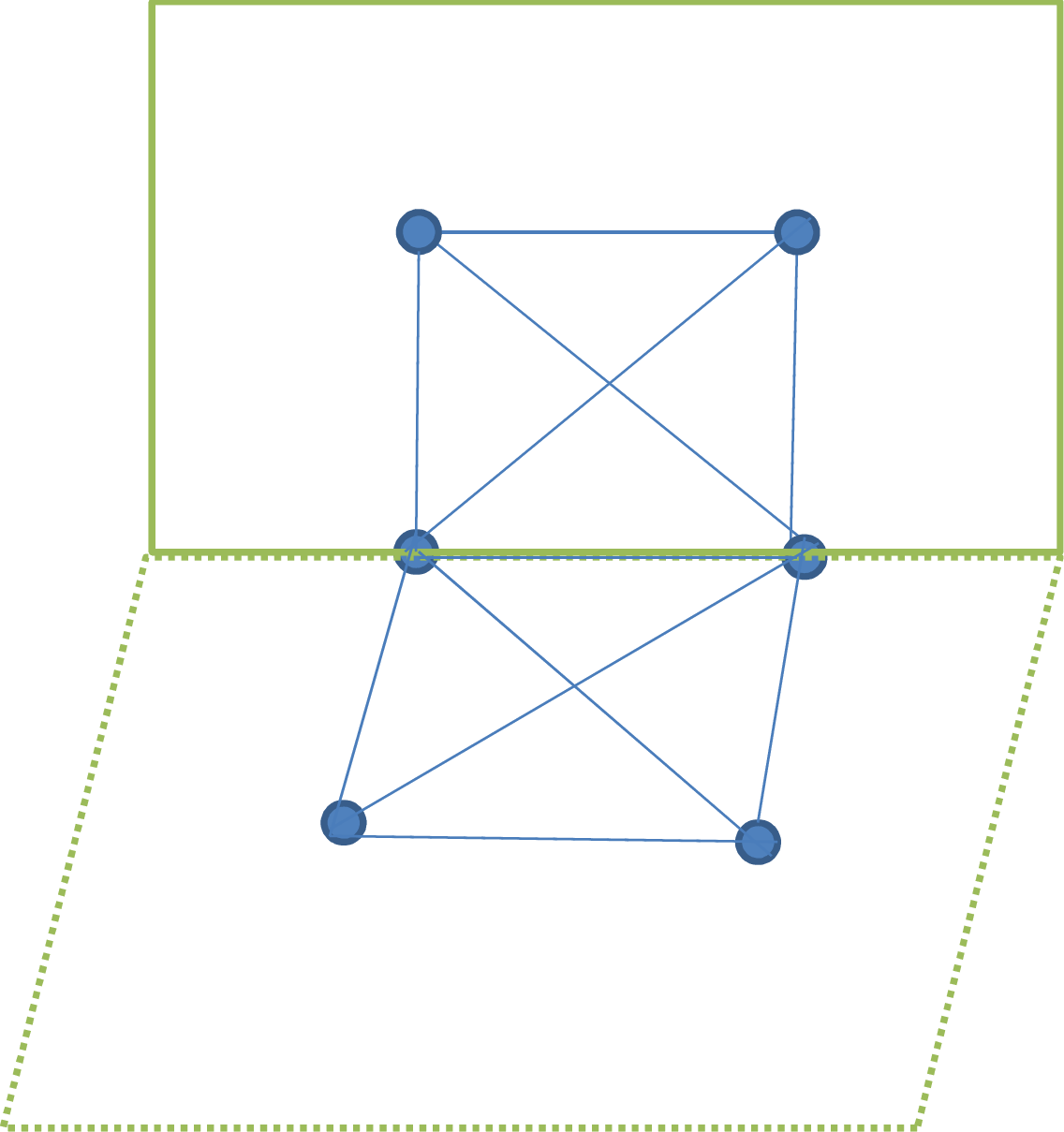}%
\end{center}
\caption{
(Left)
A ruled framework on a hyperbolic paraboloid
that is neighborhood affine
rigid. By adding six long bracing edges along the rulings,
as in Figure~\ref{fig:ex}, it
will have a PSD equilibrium stress matrix
of rank $n-d-1$.
(Middle)
A ruled framework on an elliptical cone
that is neighborhood affine
rigid.
It has one vertex at a cone point.
Even with bracing edges along the rulings, it will not
have an equilibrium stress matrix of rank $n-d-1$.
The green ellipse is added for visualization context only.
(Right)
A ruled framework on two planes
that is neighborhood affine
rigid and has
a PSD equilibrium stress matrix of rank $n-d-1$.
It has two vertices at cone points, along the intersection line of the two
planes.
The green planes are added for visualization context only.
}
\label{fig:3d}
\end{figure}
In $\EE^2$ a ruled framework must lie in the intersection of two
lines. See Figure~\ref{fig:ex}.

What do neighborhood affine rigid, ruled frameworks in
$\EE^3$ look like? See Figure~\ref{fig:3d}.

One possibility in $\EE^3$ is a framework on a doubly ruled
quadric such as a hyperbolic paraboloid or
a hyperboloid of one sheet (smooth, rank $4$).
In this case, each vertex can have at most a two
dimensional neighborhood affine span.

\begin{proposition}
With an appropriate vertex and edge set,
we can construct a framework on any doubly ruled quadric
$\QQQ$
in
$\EE^3$ that
has a PSD equilibrium stress matrix of rank $n-4$
(and thus is also neighborhood affine rigid).
\end{proposition}
\begin{proof}
The construction is as follows.  Start with a
collection of lines $\{\ell_1,\ldots, \ell_s,m_1,\ldots,m_t\}$
on $\QQQ$ so that: (1) the $\ell_i$ are in one ruling and then $m_j$
in the other; (2) the (necessarily bipartite)
intersection graph of the $\ell_i$ and
$m_j$ is connected and has minimum degree $3$; (3) $\ell_1$ and
$\ell_2$ intersect all the $m_j$.  Now we
construct a framework by putting vertices at the intersection
points of lines and all the edges between vertices
on the same line.  Call this framework $(G,\p)$.

Any framework of a complete graph with $3$ or more vertices, all
on a line,  supports a PSD equilibrium stress
matrix $\Omega$ with corank $2$.  Taking a positive linear
combination of these shows that $(G,\p)$ carries a PSD
stress matrix $\Omega$, that forces any framework
$(G,\q)$ in its kernel to have the vertices partitioned
into the same collinear subsets.

To finish up, we observe that $\ell_1$ and $\ell_2$
span at most a $3$-dimensional affine space.  Since
the $m_i$ intersect $\ell_1$ and $\ell_2$, they are
in the same space.  Each of the $\ell_i$, $i\ge 3$ intersect
at least $2$ of the $m_i$, establishing that any such kernel framework, $(G,\q)$, has at most $3$-dimensional span. Equivalently, $\Omega$ has rank $n-4$.
\end{proof}

Another possibility in $\EE^3$ is a framework on a
quadric with a single cone point, such as
an elliptic cone (rank $3$).
In this case, there can be at most one vertex with a full
dimensional neighborhood affine span.
Any other vertex can only have a one dimensional
neighborhood affine span.
With an appropriate edge set,
we can construct such a framework that is
neighborhood affine rigid.
But we can not construct such a framework to have a
equilibrium stress matrix of rank $n-d-1$
unless the entire
framework lies on three intersecting lines!

The last possibility is a framework contained entirely within two intersecting
planes (rank $2$).
We can construct such a framework that is
neighborhood affine rigid or even to have a
PSD equilibrium stress matrix of rank $n-d-1$.
A framework contained within two planes
can have at most two points in general affine position,
that have neighborhoods with full three dimensional affine
spans. (Such two points
must be on the line of intersection between the two planes.)

\section*{Acknowledgments}
RC is partially supported by NSF grant
DMS-1564493. SJG is partially supported by NSF grant  DMS-1564473.

\bibliographystyle{abbrvnat}

\begin{thebibliography}{19}
\providecommand{\natexlab}[1]{#1}
\providecommand{\url}[1]{\texttt{#1}}
\expandafter\ifx\csname urlstyle\endcsname\relax
\providecommand{\doi}[1]{doi: #1}\else
\providecommand{\doi}{doi: \begingroup \urlstyle{rm}\Url}\fi

\bibitem[Alfakih(2007)]{Alfakih-dim-rigidity}
A.~Y. Alfakih.
\newblock On dimensional rigidity of bar-and-joint frameworks.
\newblock \emph{Discrete Appl. Math.}, 155\penalty0 (10):\penalty0 1244--1253,
2007.
\newblock ISSN 0166-218X.
\newblock \doi{10.1016/j.dam.2006.11.011}.
\newblock URL \url{http://dx.doi.org/10.1016/j.dam.2006.11.011}.

\bibitem[Alfakih(2014)]{Alfakih-fields}
A.~Y. Alfakih.
\newblock Local, dimensional and universal rigidities: a unified {G}ram matrix
approach.
\newblock In \emph{Rigidity and symmetry}, volume~70 of \emph{Fields Inst.
Commun.}, pages 41--60. Springer, New York, 2014.
\newblock \doi{10.1007/978-1-4939-0781-6_3}.
\newblock URL \url{http://dx.doi.org/10.1007/978-1-4939-0781-6_3}.

\bibitem[Alfakih(2015)]{spec}
A.~Y. Alfakih.
\newblock Universal rigidity of bar frameworks via the geometry of
spectrahedra.
\newblock Preprint, arXiv:1504.00578, 2015.
\newblock URL \url{https://arxiv.org/abs/1504.00578}.

\bibitem[Alfakih and Nguyen(2013)]{alfakih2013affine}
A.~Y. Alfakih and V.-H. Nguyen.
\newblock On affine motions and universal rigidity of tensegrity frameworks.
\newblock \emph{Linear Algebra Appl.}, 439\penalty0 (10):\penalty0 3134--3147,
2013.
\newblock ISSN 0024-3795.
\newblock \doi{10.1016/j.laa.2013.08.016}.
\newblock URL \url{http://dx.doi.org/10.1016/j.laa.2013.08.016}.

\bibitem[Connelly(1982)]{Connelly-energy}
R.~Connelly.
\newblock Rigidity and energy.
\newblock \emph{Invent. Math.}, 66\penalty0 (1):\penalty0 11--33, 1982.
\newblock ISSN 0020-9910.
\newblock \doi{10.1007/BF01404753}.
\newblock URL \url{http://dx.doi.org/10.1007/BF01404753}.

\bibitem[Connelly(1993)]{Connelly-rigidity}
R.~Connelly.
\newblock Rigidity.
\newblock In \emph{Handbook of convex geometry, {V}ol.\ {A}, {B}}, pages
223--271. North-Holland, Amsterdam, 1993.

\bibitem[Connelly(2005)]{Connelly-global}
R.~Connelly.
\newblock Generic global rigidity.
\newblock \emph{Discrete Comput. Geom.}, 33\penalty0 (4):\penalty0 549--563,
2005.
\newblock ISSN 0179-5376.
\newblock \doi{10.1007/s00454-004-1124-4}.
\newblock URL \url{http://dx.doi.org/10.1007/s00454-004-1124-4}.

\bibitem[Connelly(2013)]{Shaping-space}
R.~Connelly.
\newblock Tensegrities and global rigidity.
\newblock In \emph{Shaping space}, pages 267--278. Springer, New York, 2013.
\newblock \doi{10.1007/978-0-387-92714-5_21}.
\newblock URL \url{http://dx.doi.org/10.1007/978-0-387-92714-5_21}.

\bibitem[Connelly and Gortler(2015{\natexlab{a}})]{Connelly-Gortler}
R.~Connelly and S.~J. Gortler.
\newblock Iterative {U}niversal {R}igidity.
\newblock \emph{Discrete Comput. Geom.}, 53\penalty0 (4):\penalty0 847--877,
2015{\natexlab{a}}.
\newblock ISSN 0179-5376.
\newblock \doi{10.1007/s00454-015-9670-5}.
\newblock URL \url{http://dx.doi.org/10.1007/s00454-015-9670-5}.

\bibitem[Connelly and Gortler(2015{\natexlab{b}})]{connelly2015universal}
R.~Connelly and S.~J. Gortler.
\newblock Universal rigidity of complete bipartite graphs.
\newblock Preprint, arXiv:1502.02278, 2015{\natexlab{b}}.
\newblock URL \url{https://arxiv.org/abs/1502.02278}.

\bibitem[Connelly and Whiteley(2010)]{coning}
R.~Connelly and W.~J. Whiteley.
\newblock Global rigidity: the effect of coning.
\newblock \emph{Discrete Comput. Geom.}, 43\penalty0 (4):\penalty0 717--735,
2010.
\newblock ISSN 0179-5376.
\newblock \doi{10.1007/s00454-009-9220-0}.
\newblock URL \url{http://dx.doi.org/10.1007/s00454-009-9220-0}.

\bibitem[Godsil(2001)]{god}
C.~Godsil.
\newblock The {C}olin de {V}erdière number.
\newblock Online notes, 2001.
\newblock URL \url{http://quoll.uwaterloo.ca/mine/Notes/cdv.pdf}.

\bibitem[Gortler and Thurston(2014)]{Gortler-thurston2}
S.~J. Gortler and D.~P. Thurston.
\newblock Characterizing the universal rigidity of generic frameworks.
\newblock \emph{Discrete Comput. Geom.}, 51\penalty0 (4):\penalty0 1017--1036,
2014.
\newblock ISSN 0179-5376.
\newblock \doi{10.1007/s00454-014-9590-9}.
\newblock URL \url{http://dx.doi.org/10.1007/s00454-014-9590-9}.

\bibitem[Gortler et~al.(2013)Gortler, Gotsman, Liu, and
Thurston]{Gortler-affine-rigidity}
S.~J. Gortler, C.~Gotsman, L.~Liu, and D.~Thurston.
\newblock On affine rigidity.
\newblock \emph{JoCG}, 4\penalty0 (1):\penalty0 160--181, 2013.
\newblock URL \url{http://jocg.org/index.php/jocg/article/view/49}.

\bibitem[Harris(2013)]{harris2013algebraic}
J.~Harris.
\newblock \emph{Algebraic geometry: a first course}, volume 133.
\newblock Springer Science \& Business Media, 2013.

\bibitem[Laurent and Varvitsiotis(2014)]{laurent2014positive}
M.~Laurent and A.~Varvitsiotis.
\newblock Positive semidefinite matrix completion, universal rigidity and the
strong {A}rnold property.
\newblock \emph{Linear Algebra Appl.}, 452:\penalty0 292--317, 2014.
\newblock ISSN 0024-3795.
\newblock \doi{10.1016/j.laa.2014.03.015}.
\newblock URL \url{http://dx.doi.org/10.1016/j.laa.2014.03.015}.

\bibitem[Lov\'asz(2001)]{lovasz2001steinitz}
L.~Lov\'asz.
\newblock Steinitz representations of polyhedra and the {C}olin de {V}erdi\`ere
number.
\newblock \emph{J. Combin. Theory Ser. B}, 82\penalty0 (2):\penalty0 223--236,
2001.
\newblock ISSN 0095-8956.
\newblock \doi{10.1006/jctb.2000.2027}.
\newblock URL \url{http://dx.doi.org/10.1006/jctb.2000.2027}.

\bibitem[van~der Holst(2002)]{van2002graphs}
H.~van~der Holst.
\newblock Graphs with magnetic {S}chr\"odinger operators of low corank.
\newblock \emph{J. Combin. Theory Ser. B}, 84\penalty0 (2):\penalty0 311--339,
2002.
\newblock ISSN 0095-8956.
\newblock \doi{10.1006/jctb.2001.2087}.
\newblock URL \url{http://dx.doi.org/10.1006/jctb.2001.2087}.

\bibitem[van~der Holst et~al.(1999)van~der Holst, Lov\'asz, and
Schrijver]{colin}
H.~van~der Holst, L.~Lov\'asz, and A.~Schrijver.
\newblock The {C}olin de {V}erdi\`ere graph parameter.
\newblock In \emph{Graph theory and combinatorial biology ({B}alatonlelle,
1996)}, volume~7 of \emph{Bolyai Soc. Math. Stud.}, pages 29--85. J\'anos
Bolyai Math. Soc., Budapest, 1999.

\end{thebibliography}

\end{document}